\documentclass[11pt,a4paper]{article}
\usepackage{graphicx}

\usepackage[margin=90pt]{geometry}
\usepackage{amsfonts}
\usepackage{indentfirst}             
\usepackage[perpage,symbol]{footmisc}
\usepackage[sf]{titlesec}            
\usepackage{titletoc}                
\usepackage{fancyhdr}                
\usepackage{type1cm}                 
\usepackage{float}  				 
\usepackage{indentfirst}             
\usepackage{makeidx}                 
\usepackage{textcomp}                
\usepackage{layouts}                 
\usepackage{bbding}                  
\usepackage{cite}                    
\usepackage{color,xcolor}            
\usepackage{listings}                
\usepackage{graphics}
\usepackage{graphicx}
\usepackage{epsfig}
\usepackage{amsmath}
\usepackage{amssymb}
\usepackage{amstext}
\usepackage{subfigure}
\usepackage{amsthm}
\usepackage[runin]{abstract}
\usepackage{tikz}
\usetikzlibrary{patterns}
\usepackage{arydshln}

\usepackage{tikz}
\usetikzlibrary{patterns}
\usetikzlibrary{shapes.geometric}
\usetikzlibrary{shapes,snakes}

\tikzstyle{test}=[regular polygon,regular polygon sides=3, draw=red!100,thin, fill=red!5, scale=0.13]

\numberwithin{equation}{section}

\newtheorem{definition}{Definition}[section]
\newtheorem{theorem}[definition]{Theorem}
\newtheorem{lemma}{Lemma}[section]

\newtheorem{remark}{Remark}[section]

\newtheorem{prop}[definition]{Proposition}

\usepackage[titletoc]{appendix}

\newcommand{\be}{\begin{equation}}
\newcommand{\ee}{\end{equation}}

\begin{document}
\title {Sidewise profile control of 1-d waves}

\author{Ye\c{s}im Sara\c{c} \thanks{Department of Mathematics, Atat\"{u}rk University, 25240 Erzurum, Turkey. (ysarac@atauni.edu.tr)} \, and Enrique Zuazua \thanks{Chair of Dynamics, Control and Numerics,  Alexander von Humboldt-Professorship, Department of Data Science, Friedrich-Alexander University, Erlangen-N\"{u}rnberg, 91058 Erlangen, Germany.\newline (enrique.zuazua@fau.de)} 
\thanks{Chair of Computational Mathematics, Fundaci\'{o}n Deusto, University of Deusto, 48007 Bilbao, Basque Country, Spain.} 
\thanks{Departamento de Matem\'{a}ticas, Universidad Aut\'{o}noma de Madrid, 28049 Madrid, Spain. } 
 ~  }

\date{}
\maketitle

\begin{abstract}
We analyze the sidewise controllability for the variable coefficients one-dimensional wave equation. The control is acting on one extreme of the string with the aim that  the solution tracks a given path or profile at the other free end. This sidewise profile control problem is also often referred to as nodal profile or tracking control. The problem is reformulated as a dual observability property for the corresponding adjoint system, which is proved by means of sidewise  energy propagation arguments  in a sufficiently large time, in the class of $BV$-coefficients. We also present a number of open problems and perspectives for further research.
\bigskip

\hspace{-0.2in}\textbf{Key words}~  1-d wave equations, sidewise profile controllability, nodal profile control, BV-coefficients, sidewise observability, sidewise energy estimates.

\bigskip
\hspace{-0.2in}\textbf{2020 Mathematics Subject Classification.}~Primary: 35L05, 93B05. \\

(Communicated by Boris S. Mordukhovich)
\end{abstract}

\section{Introduction}
Consider the following variable coefficients controlled 1-d wave equation:  
\begin{equation}\label{e1}
\left\{~
\begin{aligned}
&\rho (x)y_{tt}-(a(x)y_{x})_{x}=0,  && 0< x< L,~ 0< t<T\\
&y(x,0)=y_0(x), ~y_{t}(x,0)=y_1(x),&& 0< x< L\\
&y(0,t)=u(t), ~ y(L,t)=0,                 &&0<t<T.
\end {aligned}
\right.
\end{equation} \\
In \eqref{e1}, $0<T<\infty$ stands for the length of the time-horizon, $L$ is the length of the string where waves propagate, $y=y(x,t)$ is the state and $u=u(t)$ is a control that acts on the system through the extreme $x = 0$.\\

We assume that the coefficients $\rho$ and $a$ are in $BV$  and to be uniformly bounded above and  below by positive constants, \textit {i.e.}
\begin{equation}\label{e42}
0<\rho_0 \le \rho (x) \le \rho_1,~~0<a_0 \le a (x) \le a_1 ~~ \text{a.e. in} ~~(0,L)
\end{equation}
and
\begin{equation}\label{e43}
\rho, a \in BV(0,L).
\end{equation}

The main goal of this paper is to analyze the sidewise boundary controllability of \eqref{e1}. More precisely, we want to solve the following problem: {\it Given a time-horizon $T>0$,  initial data $y_0(x)$,  $y_1(x)$ and a target $p(t)$ for the flux at $x=L$, to find $u(t)$ such that the corresponding solution of the system \eqref{e1} satisfies 
\begin{equation}\label{e2}
y_{x}(L,t)=p(t), \quad t \ge 0
\end{equation}
in a time-subinterval of $[0, T]$ to be made precise and under suitable conditions on $T$, according to the velocity of propagation of waves.
}

In other words, the string being fixed at the right extreme $x=L$ and the control $u=u(t)$ acting on the left-extreme $x=0$, given a profile $p=p(t)$, we want to choose the control so that the tension of the string at $x=L$, namely $y_x(L, t)$, tracks the profile $p=p(t)$. This property will be therefore referred to as sidewise profile controllability or, simply, as sidewise controllability.

Note, however, that, because of the finite-velocity of propagation one does not expect this result to hold for all $T>0$, but rather only for $T>\tau$ large enough, so that the action of the control at $x=0$ can reach the other extreme $x=L$ along characteristics, $\tau >0$ being this waiting. For the same reasons, one does not expect the condition $y_{x}(L,t)=p(t)$ to hold only for $t\ge \tau$.

This is a non-standard controllability problem since, most often, controllability refers to the possibility of steering the solution to a target in the final time $t=T$. But here the aim is rather to assure that a given trace, the given profile $p=p(t)$,  is achieved on the boundary. 

There is an extensive literature on the controllability of wave-like equations  (see, for instance,  \cite{L1, L2, M-E, XE, E2}). Most techniques to handle these problems rely on the dual observability problem for the adjoint uncontrolled wave equation, that is then addressed using methods such as multipliers, microlocal analysis or Carleman inequalities, among others. The one-dimensional case is particularly well understood and sharp results can be obtained using sidewise energy estimates in the class of $BV$-coefficients of bounded variation, \cite{Cara}.  In that context, $BV$ is the minimal requirement on the regularity of the coefficients since counterexamples can be built in the class of H\"older continuous ones \cite{CZ}. In case coefficients are slightly less regular than $BV$ one may obtain weaker controllability properties in the sense that initial data to be controlled  needs to be smoother than expected according to the $BV$ frame, \cite{FZ1}. 

Sidewise control problems have also been previously investigated. For instance, motivated by applications on gas-flow on networks, Gugat et al. \cite {Gugat} proposed the so-called  nodal profile control problem,  the goal being  to assure that the state fits a given profile on one or some nodes of the network, after a waiting time, by means of boundary controls. In \cite{Li1}, \cite{Libook}, \cite{Wang2} and \cite{ZLL} this analysis was extended  to 1-D quasilinear hyperbolic systems by a constructive method employing the method of characteristics. 

In this paper we address this sidewise or nodal-controllability problem in the context of the 1-D wave equation above. Our two main contributions are as follows:
\begin{itemize}
\item The first one is of a methodological nature. We introduce the dual version of this problem, which leads to a non-standard observability inequality for the adjoint wave equation, involving  a non-homogeneous boundary condition at $x=L$, that needs to be estimated out of measurements done at $x=0$. This is inspired on the classical duality approach to exact controllability as introduced in \cite{L1} (see also \cite{M-E}).

As we shall see, this duality method can be applied also in several space dimensions and for other models, such as heat or Schr\"odinger equations, or when defined on networks. The nodal or profile controllability is therefore systematically reduced to proving suitable observability estimates.

\item In the particular case under consideration \eqref{e1}, we prove these observability inequalities using sidewise energy estimates.
 \end{itemize}

The combination of these two contributions allows for a rather complete understanding of the problem under consideration for 1-D waves. But the development of techniques allowing to handle the corresponding observability inequalities in the multi-dimensional context seems to be a challenging problem.

The parabolic companion of this analysis has been recently developed in \cite{Barcena} where the sidewise profile controllability of the 1-D heat equation is solved using the same duality principle and  flatness methods, \cite{MRR}. In the parabolic setting, contrarily to the wave models considered in this paper, results hold in an arbitrary  small time horizon. 

Summarizing, the main sidewise controllability result is as follows (see Figure \ref{fig:fig1}):
\begin{theorem}
Let us consider system \eqref{e1} with coefficients satisfying the assumptions
\eqref{e42} and \eqref{e43}.

Let $T>L\beta$ with
\begin{equation}\label{e47}
\beta=\operatorname{ess}\sup\limits_{x\in[0, L]}\sqrt\frac{\rho}{a}.
\end{equation}

Then, for any $p\in H^{-1}_*(L\beta, T)$ there exists a control $u\in L^2(0, T)$  such that the solution of \eqref{e1} satisfies
$$
y_{x}(L,t)=p(t) ~~ \text{for all}~~ t \in (L\beta,T).
$$
\end{theorem}

\begin{remark} \label{clardual} The space $H^{-1}_*(L\beta, T)$ stands for the dual of $H^{1}_*(0,T)$ which is constituted by the subspace of the space $H^{1}(0,T)$ constituted by the functions vanishing in the time sub-interval $(0, L\beta)$. The dual is taken with respect to the pivot space $L^2(0,T)$ .

Note that, in particular, every distribution of the form 
\begin{equation}\label{particularform}
p=\frac{d}{dt} (\varphi(t) q(t))
\end{equation} 
where $q \in L^2(L\beta, T)$,  $\varphi$ being a smooth function vanishing at $t=T$, belongs to the dual $H^{-1}_*(L\beta, T)$.

The duality pairing between $H^{1}_*(0,T)$ and $H^{-1}_*(L\beta, T)$ will be denoted by  $<\cdot, \cdot>_*$.  In the particular case in which $p$ is of the form \eqref{particularform} the duality pairing can be rewritten as follows:
\begin{equation}
<s_0, p>_*=<s_0(t),d (\varphi(t) q(t))/dt>_* =- \int_{L\beta}^T s_0'(t) \varphi(t) q(t) dt.
\end{equation}

\end{remark}

\begin{figure}[H]
	\centering
		\includegraphics[width=0.45\textwidth]{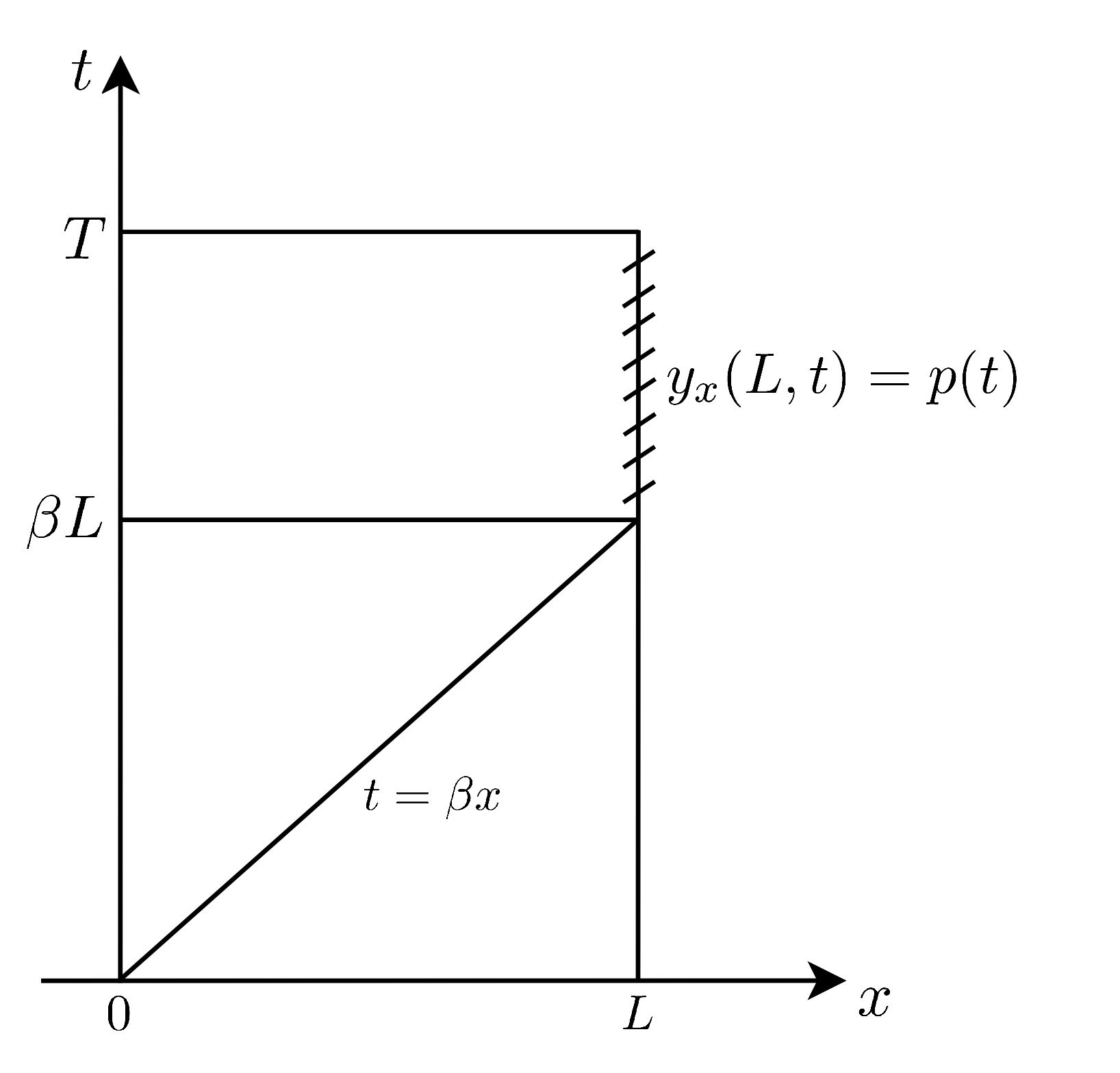}
	\caption{Sidewise controllability: The figure, which corresponds to the constant coefficient case, represents the time needed for characteristics emanating from $x=0$ to reach $x=L$, and to assure that the control of the trace to the given profile $p=p(t)$ is achieved for $T\ge L\beta$. The situation is similar for variable coefficients. The minimal time in that case corresponds to the arrival time of the curved characteristic emanating from $x=0$ at $t=0$ to reach the other extreme $x=L$.}
	\label{fig:fig1}
\end{figure}

This paper is organized as follows.  In Section 2 we present the dual sidewise observability inequality problem for the adjoint system.  We also present the main sidewise controllability result and show what the dual equivalent in terms of sidewise observability is. Section 3 is devoted to proving the relevant sidewise controllability inequality. In Section 4 we discuss the constant coefficient case showing how the same result can be achieved by the method of characteristics. We investigate the sidewise control problem for the multi-dimensional wave equation in Section 5. We conclude in Section 6 with some conclusions and some open problems for future research.

\section{The dual sidewise observability problem} 

Let us consider system \eqref{e1} with coefficients satisfying the assumptions
\eqref{e42} and \eqref{e43}.

For any given $(y_0,y_1)$ with $y_0 \in L^2(0,L)$ and $\rho y_1 \in H^{-1}(0,L)$ and any $u \in L^2(0,T)$, system \eqref{e1} admits a unique solution $y$, defined by transposition (see \cite {Cara}), enjoying the regularity property
\begin{equation}\label{transpo}
y \in C([0,T];L^2(0,L)), ~~ \rho y_{t} \in C([0,T];H^{-1}(0,L)).
\end{equation}

In this functional setting the sidewise controllability problem can be formulated more precisely as follows: {\it Given a finite time $T>L\beta$ and $p \in H^{-1}_*(L\beta,T)$ we aim to determine $u \in L^2(0,T)$ such that the solution $y$ of system \eqref{e1} satisfies the condition
\begin{equation}\label{e31}
y_{x}(L,t)=p(t) ~~ \text{for all}~~ t \in (L\beta,T).
\end{equation}}

\begin{remark}
Several remarks are in order:
\begin{itemize}
\item 
Note that this problem makes sense since solutions of \eqref{e1} in the above regularity class fulfill the added boundary regularity condition 
\begin{equation}\label{boundreg}
y_{x}(L,t) \in H^{-1}(0, T).
\end{equation}

\item Note also that in the present formulation of the sidewise controllability problem the velocity of propagation plays an important role. On one hand the sidewise controllability property is only guaranteed when $T>L\beta$. This is the natural minimal time to achieve such a result since, otherwise, because of the finite velocity of propagation, the action on $x=0$ will not reach the extreme $x=L$. On the other hand, the tracking condition is only assured in the time sub-interval $(L\beta,T)$.
\item Without loss of generality, using the principle of additive superposition of solutions of linear problems, the problem can be reduced to the particular case $y_0(x)\equiv y_1(x)\equiv 0$. 
\end{itemize}
\end{remark}

Let us now consider the adjoint system:
\begin{equation}\label{e45}
\left\{~
\begin{aligned}
&\rho (x)\psi_{tt}-(a(x)\psi_{x})_{x}=0, && 0< x< L,~ 0< t<T\\
&\psi(x,T)=0, ~ \psi_{t}(x,T)=0,            && 0< x<L\\
&\psi(0,t)=0, ~ \psi(L,t)=s(t),                && 0< t<T
\end {aligned}
\right.
\end{equation} \\
where the boundary data are of the form 
 \begin{equation}\label{e46}
s(t)=\left\{
\begin{aligned}
&s_0(t), ~~~ L\beta\le  t \le T\\
&0, ~~~~~~~~ 0\le t \le L\beta\\
\end{aligned}
\right.
\end{equation}
with $s_0\in H^1(L\beta, T)$, $s_0(L\beta)=0$.\\

This system admits an unique finite-energy solution $\psi$ such that 
\begin{equation*}
(\psi, \frac{\partial \psi}{\partial {t}}) \in C([0,T],H^{1}(0,L) \times L^2(0,L))
\end{equation*}
and
\begin{equation*}
\psi_{x}(0,.) \in L^2(0,L). 
\end{equation*}

The sidewise  observability inequality that, as we shall see, this adjoint system fulfills, and which is equivalent to the sidewise controllability problem under consideration, is the following:
\begin{prop}\label{pro41}
Let $T>L\beta$ ($\beta$ being given as in \eqref{e47}). 

Then, there exists $C_1>0$ such that the observability inequality
\begin{equation}\label{e48}
\left \Vert s_0(t)\right\Vert
_{H^{1}\left(L\beta,T\right) } \leq C_1 \left \Vert \psi_{x}(0,t)\right\Vert
_{L_{2}\left(0,T\right) } 
\end{equation}
is satisfied for every finite energy solution of \eqref{e45}-\eqref{e46}, i.e. for all $s_0 \in H^{1}(L\beta,T)$ such that $s_0(L\beta)=0$.
\end{prop}
\begin{remark}
All terms in the inequality \eqref{e48} make sense. In fact, the sidewise energy estimates in \cite{Cara} allow showing that there exists a constant $C_2=C_2(\rho,a,T)>0$ such that the finite energy solution $\psi$ \eqref{e45} satisfies 
\begin{equation}\label{e91}
 \left \Vert \psi_{x}(0,t)\right\Vert
_{L_{2}\left(0,T\right) } \leq C_2 \left \Vert s_0(t)\right\Vert
_{H^{1}\left(L\beta,T\right) }
\end{equation}
for all $s_0 \in H^{1}(L\beta,T)$ such that $s_0(L\beta)=0$.

\end{remark}

This observability inequality \eqref{e48} is equivalent to the sidewise controllability property.
In fact, out of the observability property above, one can obtain the sidewise control of minimal $L^2(0, T)$-norm by a variational principle that we present now.

 Let us consider the continuous and strictly convex quadratic functional
 \begin{equation*}
J:H^{1}_*(0,T) \longrightarrow \mathbb{R}
\end{equation*}
defined as
\begin{equation}\label{e410}
J(s_0)=\frac{1}{2}\int_{0}^{T}[(a\psi_{x})(0,t)]^2dt-a(L) < s_0(t), p(t)>_*
\end{equation}
for  $s_0\in H^{1}_*(0,T)$ (see Remark \ref{clardual} above), where $\psi$ is the solution of the adjoint system \eqref{e45}.

By $<\cdot, \cdot>_*$, in \eqref{e410}, we denote the  duality pairing between $s_0(t) \in H^{1}_*(0,T)$ and its dual $p(t) \in H^{-1}_*(L\beta,T)$ as described in Remark \ref{clardual}. To simplify the notation we will sometimes simply denote it as $\int_{0}^{T}s_0(t)p(t)dt$. But it should be taken in mind that the rigorous interpretation is given by the duality pairing. 

The observability inequality above guarantees that the functional is also coercive. The Direct method of the Calculus of Variations then ensures that $J$ has a unique minimizer.

The following lemma states that the minimum of the functional $J$ provides the desired sidewise control.
\begin{lemma}\label{lemma1}
Suppose that  $\bar s_0 \in H^{1}_*(0,T)$ is the unique minimizer of $J$ in this space. If $\bar \psi$ is the solution of the adjoint system \eqref{e45} corresponding to the minimizer $\bar s_0 $ as boundary datum, then
\begin{equation}\label{e49}
u(t)=-a(0)\bar \psi_{x}(0,t)
\end{equation}
is a control such that 
\begin{equation}
y_{x}(L,t)=p(t), \quad L\beta\le  t\le T,
\end{equation}
when the initial data $y_0\equiv y_1\equiv 0$.
\end{lemma}
\begin{remark} As mentioned above, once the control is built for $y_0\equiv y_1\equiv 0$, using the linear superposition of solutions of the wave equation, the control for arbitrary initial data can be built. The functional $J$ above can be also modified so to lead directly the control corresponding to non-trivial initial data.
\end{remark} 

\noindent \begin{proof}
Since $J$ achieves its minimum at  $\bar s_0$, the Gateaux derivative of $J$ vanishes at that point and, in other words, 
\begin{equation}\label{e37}
\begin{aligned}
&0=\lim_{h \to 0} \frac{1}{h} (J(\bar s_0+hs_0)-J(\bar s_0) )\\
&~=\int_{0}^{T}(a(0))^2\bar \psi_{x}(0,t)\psi_{x}(0,t)dt-a(L)\int_{L\beta}^{T}s_0(t)p(t)dt
\end {aligned}
\end{equation}
for any other $s_0 \in H^{1}_*(0,T)$, where $ \psi$ stands for the solution of adjoint system \eqref{e45} with $s_0$ as boundary datum.

On the other hand, multiplying the equation \eqref{e1} by the solution $\psi(x,t)$ of the adjoint problem and integrating on $(0,L)\times (0,T)$ we get (recall that, without loss of generality, we have assumed that $y_0(x)\equiv y_1(x)\equiv 0$)
\begin{equation}\label{e38}
\int_{0}^{T}a(0)u(t)\psi_{x}(0,t)dt+a(L)\int_{L\beta}^{T}s_0(t)y_{x}(L,t)dt=0.
\end{equation}

Comparing \eqref{e37} and \eqref{e38} we get that $u(t)=-a(0)\bar \psi_{x}(0,t)$  is a control which leads $y_{x}(x,t)$ to $p(t)$ on $x=L$ for all $L\beta< t<T$.
\end{proof}

\begin{remark} 
Using classical arguments it can be seen that the control obtained by this minimization method is the one of minimal $L^2(0, T)$-norm (see \cite{M-E}).
\end{remark}

\rm{\section{Proof of the sidewise observability inequality}

Let us now proceed to the proof of Proposition \ref{pro41}.}

In order to prove the observability inequality \eqref{e48}, we will use the sidewise energy estimates (see \cite{E1, Cara}) (see also \cite{CazHar}, \cite{Symes} for other applications of this technique).

We have $ T>L\beta$ and the final data at $t=T$ for $\psi$ vanish. This allows to extend the solution $\psi$ by parity or reflection with respect to $t=T$ to the interval $[0, T'=2T]$ (see Figure \ref{fig:fig2}). In this way $T'>2L\beta$.  

\begin{figure}[H]
	\centering
		\includegraphics[width=0.45\textwidth]{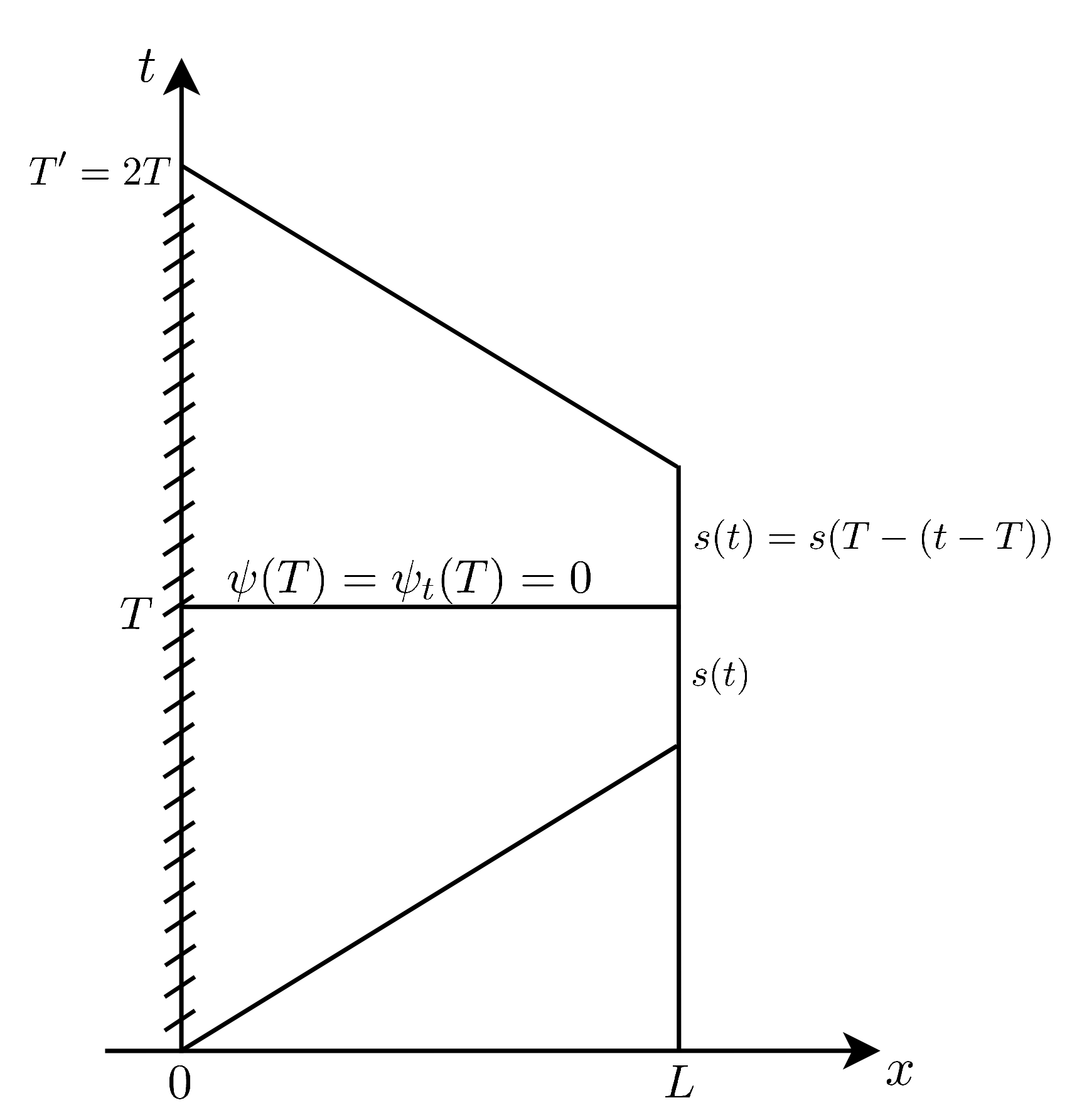}
	\caption{Extension of the solution $\psi$ (by parity with respect to $t=T$) to the interval $[0,T'=2T]$}
	\label{fig:fig2}
\end{figure}

We now proceed to perform the sidewise energy estimates in this extended time interval $[0, T'=2T]$. 
Define
\begin{equation*}
F(x)=\frac{1}{2}\int_{\beta x}^{T'-\beta x}[\rho (x) (\psi_{t}(x,t))^2+a(x)(\psi_{x}(x,t))^2]dt, ~~~~\forall x \in [0,L].
\end{equation*}
Because of the boundary condition $\psi(0,t)=0$, we have
\begin{equation*}
F(0)=\frac{a(0)}{2}\int_{0}^{T'}(\psi_{x}(0,t))^2dt  = a(0) \int_{0}^{T}(\psi_{x}(0,t))^2dt.
\end{equation*}
Let us compute the derivative of $F$ with respect to $x$:
\begin{equation}\label{e411}
\begin{aligned}
&\frac{dF(x)}{dx}=-\frac{\beta}{2} [\rho \psi_{t}^2+a\psi_{x}^2] \Big \vert_{t=T'-\beta x} -\frac{\beta}{2} [\rho \psi_{t}^2+a\psi_{x}^2] \Big \vert_{t=\beta x}\\
&~~~~~~~~~+\int_{\beta x}^{T'-\beta x} \left ( \rho (x) \psi_{t}\psi_{tx}+a(x)\psi_{x}\psi_{xx}+\frac{\rho'(x)}{2} \psi_{t}^2+\frac{a'(x)}{2}\psi_{x}^2 \right ) dt. 
\end{aligned}
\end{equation}
Integrating by parts and using the equation \eqref{e45} we have
\begin{equation*}
\begin{aligned}
&\int_{\beta x}^{T'-\beta x} \left ( \rho (x) \psi_{t}\psi_{tx}+a(x)\psi_{x}\psi_{xx} \right ) dt\\
&~~~~~~=-\int_{\beta x}^{T'-\beta x} a'(x)\psi_{x}^2 dt +\rho (x) \psi_{t} \psi_{x} \Big \vert_{t=T'-\beta x}-\rho (x) \psi_{t} \psi_{x} \Big \vert_{t=\beta x}.\\
\end{aligned}
\end{equation*}
Combining this identity and \eqref{e411} and using the elementary inequality $$\lvert \rho (x) \psi_{t} \psi_{x} \rvert \le \frac{\beta}{2} [\rho (x) \psi_{t}^2+a(x) \psi_{x}^2],$$ we obtain
\begin{equation*}
\begin{aligned}
&\frac{dF(x)}{dx} \le \frac{1}{2}\int_{\beta x}^{T'-\beta x} \left (\rho'(x) \psi_{t}^2-a'(x)\psi_{x}^2 \right ) dt\\
&~~~~~~~~~ \le  \frac{1}{2} \max \left \{ \frac{\lvert \rho' \rvert}{\rho}, \frac{\lvert a' \rvert}{a} \right \} \int_{\beta x}^{T'-\beta x} \left (\rho(x) \psi_{t}^2+a(x)\psi_{x}^2 \right ) dt\\
&~~~~~~~~~=\max \left \{ \frac{\lvert \rho' \rvert}{\rho}, \frac{\lvert a' \rvert}{a} \right \}F(x).
\end{aligned}
\end{equation*}
Integrating this differential inequality with respesct to $x$, we get
\begin{equation*}
F(x) \le \exp \left ( \int_{0}^{x}\max \left \{ \frac{\lvert \rho'(s) \rvert }{\rho(s)}, \frac{\lvert a'(s) \rvert }{a(s)} \right \}ds \right) F(0).
\end{equation*}\\
From $$\int_{0}^{x}\max \left \{ \frac{\lvert \rho'(s) \rvert }{\rho(s)}, \frac{\lvert a'(s) \rvert }{a(s)} \right \}ds \le \frac{TV(\rho)}{\rho_0}+\frac{TV(a)}{a_0},$$ we get
\begin{equation}\label{e412}
F(x) \le a(0) \exp \left ( \frac{TV(\rho)}{\rho_0}+\frac{TV(a)}{a_0} \right) \int_{0}^{T}(\psi_{x}(0,t))^2dt,  ~~~\forall x \in [0,L].
\end{equation}
Integrating with respect to $x$ in $(0,L)$, we have
\begin{equation}\label{e413}
\begin{aligned}
& \frac{1}{2}\int_{\beta L}^{T'-\beta L}\int_{0}^{L}[\rho (x) (\psi_{t}(x,t))^2+a(x)(\psi_{x}(x,t))^2]dxdt\\
&~~~~~~~~~~~\le a(0)L \exp \left ( \frac{TV(\rho)}{\rho_0}+\frac{TV(a)}{a_0} \right) \int_{0}^{T}(\psi_{x}(0,t))^2dt.
\end{aligned}
\end{equation}
Because of the boundary condition $\psi(0,t)=0$, we can write that
\begin{equation*}
\begin{aligned}
&\int_{\beta L}^{T'-\beta L}(\psi(L,t))^2dt=\int_{\beta L}^{T'-\beta L} \left \{ \int_{0}^{L}\psi_{x}(x,t)dx \right \} ^2dt\\
&~~~~~~~~~~~~~~~~~~~~~~~~~\le L \int_{\beta L}^{T'-\beta L}\int_{0}^{L}\psi_{x}(x,t)^2dxdt\\
&~~~~~~~~~~~~~~~~~~~~~~~~~\le \frac{L}{\min \{\rho_0,a_0\} }\int_{\beta L}^{T'-\beta L}\int_{0}^{L}[\rho (x) (\psi_{t}(x,t))^2+a(x)(\psi_{x}(x,t))^2]dxdt.
\end{aligned}
\end{equation*}
Combining this inequality and \eqref{e413}, we get
\begin{equation}\label{e414}
\begin{aligned}
&\int_{\beta L}^{T}(\psi(L,t))^2dt \le \frac{1}{2}\int_{\beta L}^{T'-\beta L}(\psi(L,t))^2dt\\
&~~~~~~~~~~~~~~~~~~~~ \le \frac{a(0)L^2}{\min \{\rho_0,a_0\}} \exp \left ( \frac{TV(\rho)}{\rho_0}+\frac{TV(a)}{a_0} \right) \int_{0}^{T}(\psi_{x}(0,t))^2dt.
\end{aligned}
\end{equation}
On the other hand, from \eqref{e412} we have
\begin{equation*}
F(L) \leq  a(0)\exp \left ( \frac{TV(\rho)}{\rho_0}+\frac{TV(a)}{a_0} \right) \int_{0}^{T}(\psi_{x}(0,t))^2dt.
\end{equation*}
Therefore we can write
\begin{equation}\label{e415}
\begin{aligned}
&\rho (L)\int_{\beta L}^{T}(\psi_{t}(L,t))^2dt  \le \frac{\rho (L)}{2}\int_{\beta L}^{T'-\beta L}(\psi_{t}(L,t))^2dt\\
&~~~~~~~~~~~~~~~~~~~~~~~~~~~ \leq  a(0) \exp \left ( \frac{TV(\rho)}{\rho_0}+\frac{TV(a)}{a_0} \right) \int_{0}^{T}(\psi_{x}(0,t))^2dt
\end{aligned}
\end{equation}
where
\begin{equation*}
\psi_{t}(L,t)=\left\{
\begin{aligned}
&s'_0(t), ~~~ L\beta\le  t \le T\\
&0, ~~~~~~~~ 0\le t \le L\beta\\
\end{aligned}
\right.
\end{equation*}
Combining \eqref{e414} and \eqref{e415} and taking into account \eqref{e46} we get the desired observability inequality \eqref{e48} with
\begin{equation*}
C_1^2=\left( \frac{L^2}{\min \{\rho_0,a_0\}}+\frac{1}{\rho (L)} \right )a(0) \exp \left ( \frac{TV(\rho)}{\rho_0}+\frac{TV(a)}{a_0} \right).
\end{equation*}
Thus, the proof of Proposition \ref{pro41} is done.

\section{The constant coefficients wave equation}

In this section, we consider constant coefficients 1-d wave equation, \textit {i.e.} the particular case where $\rho \equiv a \equiv 1$, to show how the method of characteristics can be implemented to obtain the previous result:  
\begin{equation}\label{eq1}
\left\{~
\begin{aligned}
&y_{tt}-y_{xx}=0,&& 0< x< L,~ 0< t<T,\\
&y(x,0)=y_0(x), ~y_{t}(x,0)=y_1(x), && 0< x< L,\\
&y(0,t)=v(t), ~ y(L,t)=0,&& 0<t<T.
\end {aligned}
\right.
\end{equation} \\

The following Theorem \ref {thm_1} can be proved by the methods of characteristics in \cite {Libook},  states the existence of a control in spaces of smooth solutions.  The same construction can also be implemented in energy spaces.

\begin{theorem}\label{thm_1}
Assume that $y_0(x)=0$ and $y_1(x)=0$. Let
\begin{equation*}
\overline T>L
\end{equation*}
and $T$ be an arbitrarily given number such that 
\begin{equation*}
T>\overline T.
\end{equation*}
For any given function $q(t)\in C^1([\overline T, T])$, we can find a boundary control $v(t)\in C^2([0,T])$ such that the system \eqref{eq1} admits a unique solution $y\in C^2\big( [0,L]\times [0,T] \big)$ satisfying $y_{x}(L,t)=q(t)$ for all $\overline T\le t\le T$.
\end{theorem}

\begin{proof}
The proof we present here is of constructive nature,  similar to those in \cite {Wang, Libook}.
We construct the solution to the control problem in the following steps. All along the proof $y_0(x)=0$ and $y_1(x)=0$. \\

 \textbf{Step 1:} Let $\overline T>L$ and $T$ be an arbitrarily given number such that $T>\overline T$. On the domain $\mathcal R(L)=\{(x,t):0\leq x \leq L,0\leq t \leq L\}$, we consider the following initial-boundary value problem for the equation \eqref{eq1} with an artificial boundary condition $y(0,t)=f(t)$
\begin{equation}\label{eq3}
\left\{~
\begin{aligned}
&y_{tt}-y_{xx}=0,&& 0< x< L,~ 0< t< L,\\
&y(x,0)=0, ~ y_{t}(x,0)=0, && 0\le x \le L,\\
&y(0,t)=f(t), ~ y(L,t)=0, && 0\le t \le L
\end{aligned}
\right.
\end{equation} \\
where $f(t)$  is an arbitrarily given $C^2([0,L])$ function, satisfying the compatibility conditions $f(0)=0, f'(0)=0$ and $f''(0)=0$.\\

By the theory on classical solutions for 1-D linear wave equations, this initial-boundary problem admits a unique solution on the domain $\mathcal R(L)$ (see Figure \ref{fig:fig4}(a)). Let  $y_f(x,t)$ denote the solution of the problem \eqref{eq3} for the boundary condition $f(t)$.

We introduce 
\begin{equation*}
 a(t)=\frac{\partial}{\partial {x}} y_f(L,t), ~~0 \le t\le L
\end{equation*}
which is in $C^1([0,L])$ and is uniquely determined by $f$ (see Figure \ref{fig:fig4}(b)).\\

\begin{figure}[t]
	\centering
		\includegraphics[width=0.8\textwidth]{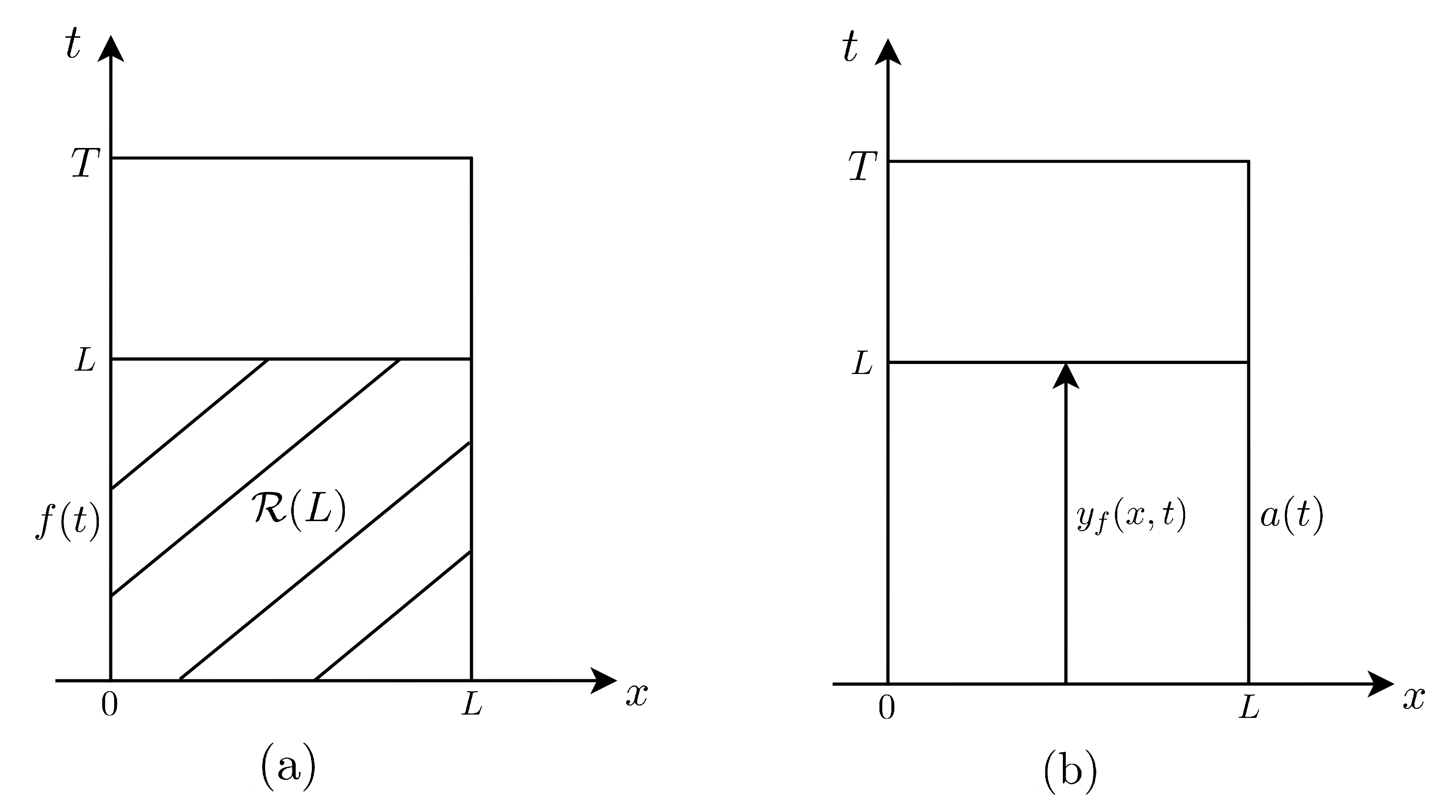}
        \caption{Construction of the first step of the proof.}
	\label{fig:fig4}
\end{figure}

Given $L>0$ and $\overline T>L$ we can define a function $c(t) \in C^1([0,T])$ so that
\begin{equation}\label{eq4}
c(t)=\left\{
\begin{aligned}
&a(t), && 0\leq t\leq L\\
&z(t), && L\leq t\leq \overline T\\
&q(t), && \overline T\leq t\leq T\\
\end{aligned}
\right.
\end{equation}
where $q$ is the given target function and $z \in C^1([L, \overline T])$ is any given function, satisfying the following conditions
\begin{equation*}
z(L)=a(L),~~~ z(\overline T)=q(\overline T),~~~  z'(L)=a'(L)~~ \text{and}~~  z'(\overline T)=q'(\overline T).
\end{equation*}

\bigskip

 \textbf{Step 2:} We change the role of $t$ and $x$ and consider a leftward initial-boundary value problem
\begin{equation}\label{eq5}
\left\{~
\begin{aligned}
&y_{xx}-y_{tt}=0, && 0< x< L,~ 0< t<T\\
&y(L,t)=0, ~y_{x}(L,t)=c(t), && 0\le t \le T,\\
&y(x,0)=0, ~ y(x,T)=\phi (x), && 0\le x\le L,
\end{aligned}
\right.
\end{equation} \\
on the domain $\mathcal R(T)=\{(x,t):0\leq x \leq L,~0\leq t \leq T\}$ where $\phi (x) \in C^2([0,L])$ is an arbitrary given function, satisfying $C^2$ compatibility conditions $\phi(L)=0$, $\phi'(L)=c(T)$ and $\phi''(L)=0$.\\

By the theory on classical solutions for 1-D linear wave equations, we know that this leftward initial-boundary value problem admits a unique solution $y=y(x,t)$ on the domain $\mathcal R(T)$ (see Figure \ref{fig:fig5}).

From definition of the function $c=c(t)$ it is clear that this solution $y=y(x,t)$ satisfies the desired condition $y_{x}(L,t)=q(t)$ for all $\overline T\le t\le T$.

\begin{figure}[H]
	\centering
		\includegraphics[width=0.4\textwidth]{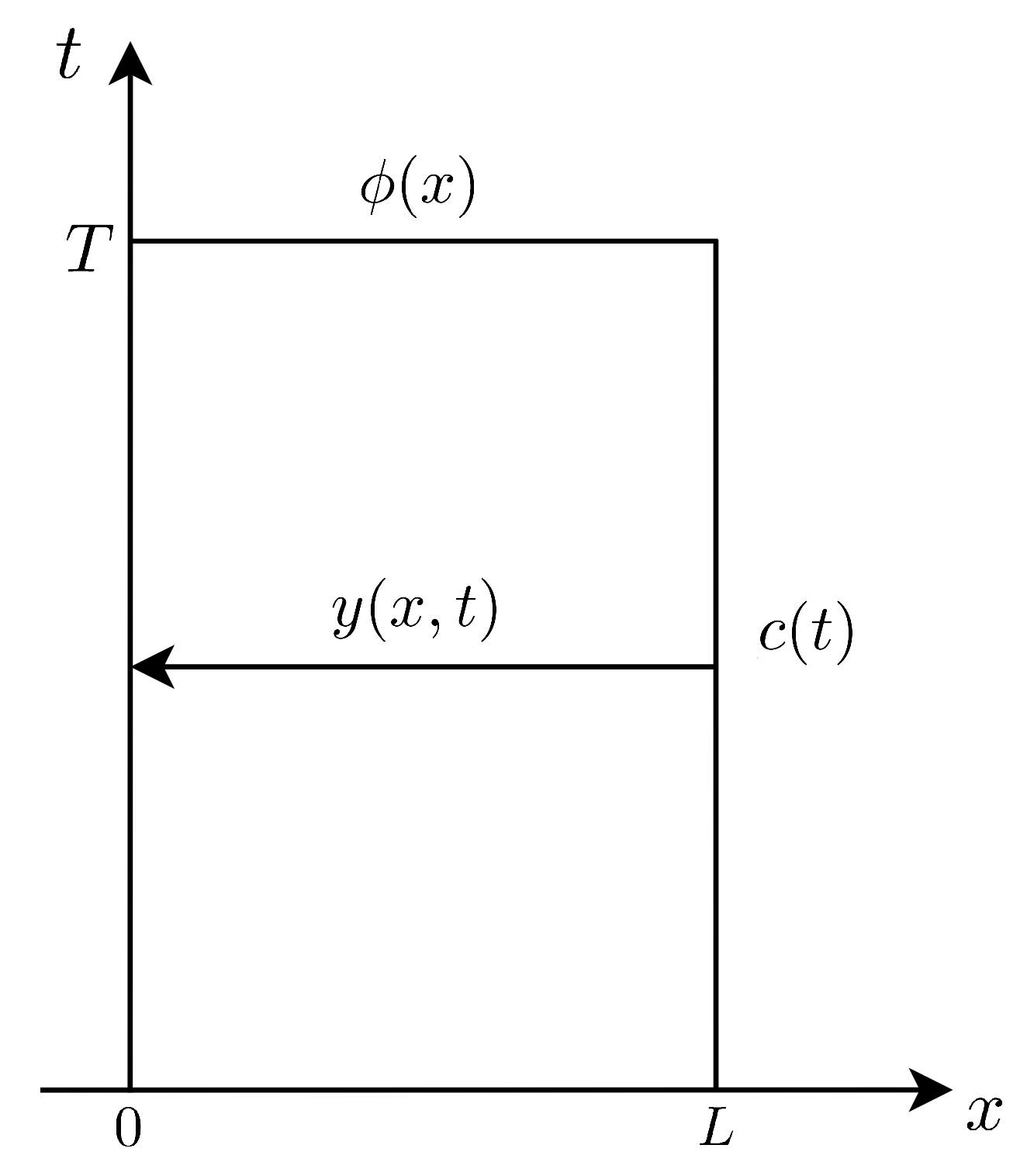}
	\caption{Leftward initial-boundary value problem on the domain $\mathcal R(T)$.}
	\label{fig:fig5}
\end{figure}
 
\bigskip

\textbf{Step3:} The solution $y=y(x,t)$ of the problem \eqref{eq5} satisfies the equation $y_{tt}-y_{xx}=0$ and the first initial condition $y(x,0)=0$ and the boundary condition $y(L,t)=0$ in the system \eqref{eq1}. We need to prove that it also satisfies the second initial condition
\begin{equation*}
y_{t}(x,0)=0.
\end{equation*}
If the second initial condition is indeed satisfied, then we get the desired control function $v$ by substituting $y=y(x,t)$ into the boundary condition in \eqref{eq1}. \\

Consider the following one-sided initial boundary value problem:
\begin{equation}\label{eq6}
\left\{~
\begin{aligned}
&y_{xx}-y_{tt}=0,&& 0< x< L,~ 0< t<x,\\
&y(L,t)=0, ~ y_{x}(L,t)=a(t), && 0\le  t\le x,\\
&y(x,0)=0,&& 0\le x\le L.
\end{aligned}
\right.
\end{equation} \\
Both the solution $y=y_f(x,t)$ of \eqref{eq3} and the solution $y=y(x,t)$ of \eqref{eq5} are solutions to the problem \eqref{eq6}. By the uniqueness of $C^2$-solution to one-sided problem on the domain $\{(x,t): 0\le x \le L, 0\le t \le x \}$ (see Figure \ref{fig:fig6}), $y( x,t)=y_f( x,t)$ on the interval $\{t=0,~0\leq x \leq L\}$. Thus, $y(x,t)$ satisfies the second initial condition $y_{t}(x,0)=0$ since $ \frac{\partial}{\partial {t}} y_f(x,0)=0$.\\

\begin{figure}[H]
	\centering
		\includegraphics[width=0.4\textwidth]{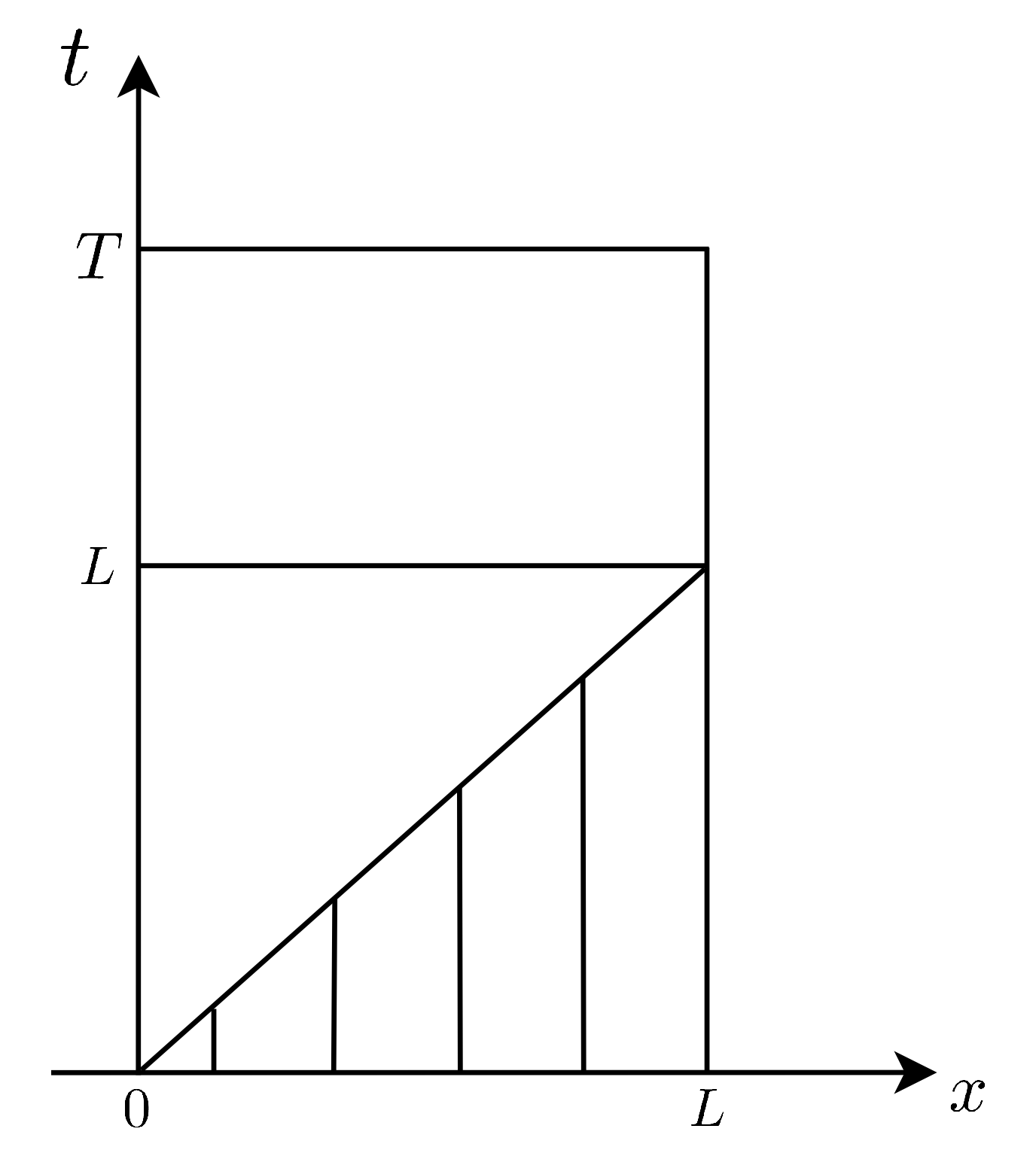}
        \caption{The domain $\{(x,t): 0\le x \le L, 0\le t \le x \}$ which is insensitive to the action of the control.}
	\label{fig:fig6}
\end{figure}

\end{proof}

\section{Multi-dimensional problems}  As explained in the introduction, the controllability theory of wave equations has been also developed in the multi-dimensional context. Duality arguments reduce the problem to boundary observability inequalities that can be obtained by different methods including multipliers, Carleman inequalities and microlocal analysis.

The sidewise control problem discussed in this paper can also be easily reformulated in the multi-dimensional frame. The duality method described here can also be applied, reducing the problem to the obtention of new sidewise observability inequalities.

However, adapting the existing techniques for the observability of waves to prove those sidewise observability inequalities for multi-dimensional wave equations seems to be a challenging problem that we present here in some more detail.

Let $\Omega$ be a bounded smooth domain of $\mathbb{R}^n$, in dimension $n\ge 2$, and consider the wave equation:
\begin{equation}\label{e1multi}
\left\{~
\begin{aligned}
&y_{tt}-\Delta y =0,  &&\hbox{ in } \Omega \times (0, T)\\
&y(x,0)=y_0(x), ~y_{t}(x,0)=y_1(x),&& \hbox{ in } \Omega\\ 
&y=u,              &&\hbox{ on } \Gamma_c\times (0, T) \\
&y=0,                 &&\hbox{ on } \Gamma_0\times (0, T).
\end {aligned}
\right.
\end{equation} 

Here $\Gamma_0$ and $\Gamma_c$ stand for a partition of the boundary, $\Gamma_0$ being the fix part of the boundary and $\Gamma_c$ the one under control. While the control $u$ acts on $\Gamma_c$, the subset of the boundary $\Gamma_0$ remains fixed, thanks to the homogeneous Dirichlet boundary conditions.

Given a smooth enough target profile $p: \Gamma_0\times(T_0, T)\to \mathbb{R}$  to be tracked,  with $T_0>0$ large enough depending on the geometry of the domain and the characteistic travel time in a way to be determined, the question is then to find a control $u$ in, say, $L^2(\Gamma_c \times (0, T))$, such that the solution $y$ of \eqref{e1multi} satisfies the condition
$$
\partial y/\partial \nu = p \quad \hbox{on} \, \Gamma_0\times(T_0, T).
$$
Here and in what follows $\nu$ denotes the outward unit normal vector and $\partial \cdot/\partial \nu$ the normal derivative.

This is the natural multi-d version of the sidewise controllability problem discussed above. In fact, the same formulation can be easily adapted to consider other models such as multi-dimensional heat, Schr\"odinger equations or the system of thermoelasticity (see \cite{survey}).

The sidewise controllability problem above can be reduced, by duality, to a new class of sidewise observability inequalities for the adjoint system
\begin{equation}\label{e1multi2}
\left\{~
\begin{aligned}
&\psi_{tt}-\Delta \psi =0,  &&\hbox{ in } \Omega \times (0, T)\\
&\psi(x,T)=\psi_{t}(x,T)=0,&& \hbox{ in } \Omega\\
&\psi=0,              &&\hbox{ on } \Gamma_c\times (0, T) \\
&\psi=s,                 &&\hbox{ on } \Gamma_0\times (0, T).
\end {aligned}
\right.
\end{equation} 
Here $s=s(x, t)$ is a smooth boundary condition given on $\Gamma_0\times (0, T)$. The question is then  whether one can prove the existence of an observability constant $C>0$ such that
\begin{equation}
||s||^2_{*} \le C ||\partial \psi/\partial \nu||_{L^2( \Gamma_c\times (0, T))}^2
\end{equation}
for every solution of this adjoint system.

Observe that here the norm $||\cdot ||_*$ in this inequality is to be identified both in what corresponds the Sobolev regularity and the support within $\Gamma_0 \times (T_0, T)$.

 As we mentioned above the existing techniques do not seem to yield this kind of inequalities in a direct manner.
 
 However, as described in \cite{L2}, using the Holmgren's uniqueness Theorem, a unique continuation property can be easily proved. This constitutes a weaker and non-quantitative version of this kind of inequality. 
 
 Indeed, it can be easily proved, using the same extension by parity with respect to $t=T$ as above,  that, as soon as $\partial \psi/\partial \nu \equiv 0$ in $\Gamma_c\times (0, T)$ and $T$ is large enough, one can guarantee that $s\equiv 0$ provided its support is localized in a subset of the boundary of the form $\gamma \times (\tau, T)$, with $\gamma$ a suitable open subset of $\Gamma_0$ and $0 < T_0 <T$. Both $\gamma$ and $T_0$  can be easily characterized it terms of the cones of influence and dependence of solutions of the wave equation. Essentially, $\gamma$ is constituted by the points for which the geodesic distance (within $\Omega$)  to $\Gamma_c$ is less than $T_0$. 
 
 Obviously, for this result to be active in some effective subset $\gamma \times (T_0, T)$, one needs $T$ to be large enough, in particular $T>T_0$, where $T_0$ is the minimal geodesic distance from $\Gamma_c$ to $\Gamma_0$.

 This unique continuation result assures, in the corresponding geometric setting, that the wave equation enjoys the property of sidewise approximate controllability: \textit {i.e.} that given any $p\in L^2(\gamma \times (T_0, T))$ and any $\epsilon >0$ there exists a control $u\in L^2(\Gamma_c \times (0, T))$ (depending on $\epsilon$) such that the corresponding solution $y$ satisfies
$$
||\partial y/\partial \nu - p ||_{L^2(\gamma \times (T_0, T))} \le \epsilon.
$$

This result can be viewed as a partial extension of the 1-D results in this paper  to the multi-dimensional case. Note however that these arguments, based purely on Holmgren uniqueness, do not yield any quantitative estimates.

The systematic analysis of these problems in the multi-dimensional context for the wave equation and other relevant models constitutes a very rich source of interesting open problems.

\section{Conclusions and other open problems}

In this paper we have proved the sidewise controllability of the 1-D wave equation with $BV$ coefficients. This was done superposing, on one hand, a dual formulation of the problem, which leads to a novel sidewise observability inequality that, on the other hand, we prove by sidewise energy estimates that have been previously developed and implemented in the context of control of 1-D wave equations.

The methods and results in this paper lead to some interesting open problems and could be extended in various directions that we briefly describe now:
\begin{enumerate}
\item{\bf Optimal control.} Instead of consider the sidewise controllability problem in this paper one could adopt a more classical optimal control approach. The problem then could be formulated as that in which one minimizes a functional of the form
$$
\frac{1}{2} \left [\int_0^T u^2(t) dt + \kappa ||y(L, t)-p(t)||_{H^{-1}(L\beta, T)}^2\right ],
$$
depending on $u\in L^2(0, T)$, with $\kappa>0$ any penalty parameter.

Optimal controls for this problem exist for all $T>0$. This is simply due to the quadratic structure of the functional to be minimized, its coercivity and continuity.

The controllability problem discussed here is a singular limit as $\kappa \to \infty$. But of course this limit process depends on the length of the time interval $T$ since, as we have seen, one can only expect the sidewise target to be reached exactly when $T>L\beta$.

For this optimal control problem, when the time horizon $T>0$  is long enough, turnpike properties were proved in \cite{Gugat2}, \cite{GTZ}. They assert that, when $T\to \infty$, the optimal control and optimal trajectories are close to the steady state ones (are time-independent), in most of the time horizon $[0, T]$, except for some exponential boundary layers at $t=0$ and $t=T$.

\item{\bf Other boundary conditions.} For the sake of simplicity, in this paper the case of Dirichlet boundary conditions has been addressed. But similar problems are relevant with other boundary conditions. Our techniques apply in those cases too.

One could for instance consider the same model with Neumann boundary conditions and control:
\begin{equation}\label{e1N}
\left\{~
\begin{aligned}
&\rho (x)y_{tt}-(a(x)y_{x})_{x}=0,  && 0< x< L,~ 0< t<T,\\
&y(x,0)=y_0(x), ~y_{t}(x,0)=y_1(x),&& 0< x< L,\\
&y_x(0,t)=u(t), ~ y_x(L,t)=0,                 &&0<t<T.
\end {aligned}
\right.
\end{equation} 
In this case the problem consists on, given a time-dependent function $p=p(t)$, to find a control $u=u(t)$ such that the corresponding solution fulfills:
\begin{equation}\label{e2N}
y(L,t)=p(t), \quad t \ge 0
\end{equation}
Our methods apply in this case too, leading to similar results with minor changes.

\item {\bf Less regular coefficients.} The methods developed in this paper combined with those of \cite{FZ1} allow to consider coefficients with slightly weaker regularity and obtain sidewise controllability results for more smooth Sobolev targets. Note, however, that the results in \cite{CZ} can also be adapted to show that in the class of H\"older continuous coefficients one cannot expect such results for targets in Sobolev classes.

\item {\bf Non-harmonic Fourier series.}The results presented in this paper could be also obtained using other genuinely 1-D methods such as the  D'Alembert formula or Fourier series representation methods. 

When dealing with Fourier series the sidewise observability inequality proved above leads to new variants of the classical Ingham inequalities (see \cite{M-E}). It would be interesting to see if they can be obtained directly by non-harmonic Fourier series methods.

\item {\bf Nonlinear problems.} The results in this paper, combined with those in \cite{E1}, allow to extend our sidewise controllability result for semilinear wave equations of the form 
$$
\rho (x)y_{tt}-(a(x)y_{x})_{x} + f(y)=0
$$
with $f$ a locally Lipschitz nonlinearity such that
$$
\lim_{|s|\to \infty} \frac{f(s)}{s\log^2(s)} 
$$
is small enough (see also \cite{Cannarsa}).

These techniques, combined with linearization techniques could also be useful to handle quasilinear problems in the regime of small amplitud smooth solutions, to give an alternative proof to results similar to those in \cite{Gugat}. But this would require further work.

\item{\bf Transmutation.} As explained in \cite{Cara}, transmutation techniques can be used to transfer controllability properties of wave equations into heat equations. It would be interesting to analyse whether this can be done in the context of the sidewise controllability/observability of the heat equation. 

Recall, as mentioned above, that the sidewise controllability of the 1-D heat equation has been directly addressed in \cite{Barcena} using flatness methods in 
Gevrey classes.

\item{\bf Networks.} The techniques in this paper can be applied for 1-D wave propagation on networks. It could for instance allow to handle the case of a tree-shaped network with active controls on all but one free ends (\cite{LLS}). But adapting them to more general networks, or to the case of fewer controls, would require substantial further developments in combination with graph and diophantine theory (\cite{Dager}).

\item{\bf Numerical analysis.} Most of the methods developed for the numerical controllability (\cite{E1}) of the wave equation can also be applied in the context of sidewise controllability. But this would require a careful adaptation since, most often, the needed numerical results are achieved using Fourier series techniques, and not the sidewise energy estimates presented here, that do not hold in the discrete setting.

\item{\bf Fourier series interpretation.} We refer to \cite{EZ2021} for various extensions of the results presented in this paper, and their interpretation in terms of harmonic and non-harmonic Fourier series, which leads to  interesting and challenging open problems.

\end{enumerate}

\bigskip

\noindent
\textbf{Acknowledgments}: Authors thank Martin Gugat (FAU) for fruitful discussions.

This work was done while the first author was visiting FAU-Erlangen during a sabbatical year. 

The second author has received funding from  the European Research Council (ERC) under the European Union’s Horizon 2020 research and innovation programme (grant agreement NO. 694126-DyCon), the European Union’s Horizon 2020 research and innovation programme under the Marie Sklodowska-Curie grant agreement No.765579-ConFlex.D.P., the Alexander von Humboldt-Professorship program, the Transregio 154, Mathematical Modelling, Simulation and Optimization using
 the Example of Gas Networks,
  of the German DFG, project C08 and by grant MTM2017-92996-C2-1-R COSNET of MINECO (Spain) .

\end{document}